\documentclass[12pt]{amsart}
\usepackage{geometry}                
\usepackage{graphicx}
\usepackage{amssymb}
\newtheorem{theorem}{Theorem}

\newtheorem{lemma}{Lemma}

\newtheorem{proposition}{Proposition}
\newtheorem{corollary}{Corollary}

\newcommand{\CC}{\mathbb{C}}

\newcommand{\ZZ}{\mathbb{Z}}
\newcommand{\RR}{\mathbb{R}}
\newcommand{\NN}{\mathbb{N}}

\newcommand{\one}{\mathbf{1}}

\title{Decoupling and near-optimal restriction estimates for Cantor sets} 
\author{Izabella {\L}aba and Hong Wang}
\date{July 27, 2016}

\begin{document}

\begin{abstract}
For any $\alpha\in(0,d)$, we construct Cantor sets in $\RR^d$ of Hausdorff dimension $\alpha$ such that the
associated natural measure $\mu$ obeys the restriction estimate $\| \widehat{f d\mu} \|_{p} \leq C_p  \| f \|_{L^2(\mu)}$
for all $p>2d/\alpha$. This range is optimal except for the endpoint. This extends the earlier work of Chen-Seeger and Shmerkin-Suomala, where a similar result was obtained by different methods for $\alpha=d/k$ with $k\in\NN$. Our proof is based
on the decoupling techniques of Bourgain-Demeter and a theorem of Bourgain on the existence of $\Lambda(p)$ sets.
\end{abstract}

\maketitle

\section{introduction}

We define the Fourier transform 
$$
\widehat{f}(\xi) =  \int e^{-2\pi i x \cdot \xi} f(x)  dx \qquad \forall \xi \in \RR^d.
$$
If $\mu$ is a measure on $\RR^d$, we will also write
$$
\widehat{f d\mu}(\xi) = \int e^{-2\pi i x \cdot \xi} f(x)  d\mu(x) \qquad \forall \xi \in \RR^d.
$$

We are interested in estimates of the form
\begin{align}\label{restriction 1}
\| \widehat{g d\mu} \|_{p} \leq C \| g \|_{L^q(\mu)} \qquad \forall g \in L^q(\mu),
\end{align}
where the constant may depend on the measure $\mu$ and on the exponents $p,q$, but not on $f$.
If $\mu$ is a probability measure, we trivially have $\| \widehat{g d\mu} \|_{\infty } \leq \| g \|_{L^1(d\mu)} \leq \| g \|_{L^q(d\mu)} $,
so that (\ref{restriction 1}) holds with $p=\infty$ and all $q\in[1,\infty]$. In general, it is not possible to say more than that. However, the problem becomes more interesting 
if we restrict attention to specific well-behaved classes of measures.

There is a vast literature on restriction estimates for smooth manifolds (see e.g. \cite{stein-ha}, \cite{tao-survey}, \cite{tao-survey2} for an overview and a selection of references). It is well known that (\ref{restriction 1}) cannot hold
with $p<\infty$ (and any $q$)
when $\mu$ is supported on a flat manifold such as a hyperplane. On the other hand, such estimates are possible if $\mu$ is the surface measure on a curved manifold $M$, with the range of exponents $p,q$ depending on the geometry of $M$, in particular on its dimension, smoothness and curvature.

In the model case when $\mu$ is the Lebesgue measure on the sphere $S^{d-1}\subset\RR^d$, the classic Tomas-Stein theorem states that (\ref{restriction 1}) holds with $q=2$ and $p\geq \frac{2d+2}{d-1}$. It was furthermore conjectured by Stein that for $q=\infty$, the range of $p$ could be improved to $p > \frac{2d}{d-1}$; this has been proved for $d=2$, but remains open in higher dimension, with the current best results due to Guth \cite{guth-poly1}, \cite{guth-poly2}.

The conjectured range $p > \frac{2d}{d-1}$ for the sphere, if true, would be the best possible. This follows by letting $f\equiv 1$ and using the well known stationary phase asymptotics for $\widehat{d\mu}$. The range of $p$ in the Tomas-Stein theorem is also known to be optimal. Here, the sharpness example is provided by the Knapp construction where $f$ is the characteristic function of a small spherical cap of diameter $\delta\to 0$.

We are interested in the case when $\mu$ is a fractal measure on $\RR^d$, singular with respect to Lebesgue. Here, again, additional assumptions are necessary to make nontrivial estimates of the form (\ref{restriction 1}) possible. For example, if $\mu$ is the natural self-similar measure on the Cantor ternary set, an easy calculation shows that (\ref{restriction 1}) cannot hold for any $p<\infty$. However, if we assume that $\mu$ obeys an additional Fourier decay condition, then the following result is known. Here and below, we use $B(x,r)$ to denote the closed ball of radius $r$ centered at $x$.

\begin{theorem}\label{mm-theorem}
Let $\mu$ be a Borel probability measure on $\RR^d$. Assume that there are $\alpha, \beta \in (0,d)$ and $C_1,C_2\geq 0$ such that 
\begin{equation}\label{A}
 \mu(B(x,r)) \leq C_1  r^{\alpha} \quad \forall x \in \RR^d, r > 0,
\end{equation}
\begin{equation}\label{B}
 |\widehat{\mu}(\xi)| \leq C_2 (1+|\xi|)^{-\beta/2} \quad \forall \xi \in \RR^d
\end{equation}
Then for all $p \geq (4d-4\alpha+2\beta)/\beta$, the estimate (\ref{restriction 1}) holds with 
$q=2$.
\end{theorem}

Theorem \ref{mm-theorem} is due to Mockenhaupt \cite{mock} and Mitsis \cite{mit} in the non-endpoint range; the endpoint was
settled later by Bak and Seeger \cite{BS}. In the case $\alpha = \beta = d-1$, this recovers the Tomas-Stein theorem for the sphere.

The range of exponents $p$ in Theorem \ref{mm-theorem} is known to be the best possible in dimension 1, in the sense that for any $0<\alpha\leq \beta <1$, there exists a probability measure $\mu$ on $\RR$, supported on a set of Hausdorff dimension $\alpha$ and obeying (\ref{A}) and ((\ref{B}), such that (\ref{restriction 1}) fails for all 
$p < (4-4\alpha+2\beta)/\beta$, see \cite{HL}, \cite{chen}. The examples are based on a construction due to Hambrook and {\L}aba \cite{HL}: the idea is to modify a random Cantor-type construction so as to embed a lower-dimensional Cantor subset that has much more arithmetic structure than the rest of the set. This can be viewed as an analogue of the Knapp example for fractal sets. A higher-dimensional variant of the construction with $d-1<\alpha\leq\beta<d$ is given in \cite{HL2}.

On the other hand, there exist specific measures on $\RR^d$ for which the range of exponents in (\ref{restriction 1}) can be improved further. If $\mu$ is supported on a set of Hausdorff dimension $\alpha<d$, it is easy to see using energy integrals that (\ref{restriction 1}) cannot hold outside of the range $p\geq 2d/\alpha$, even if $q=\infty$. (See e.g. \cite[Section 1]{HL}; the counterexample is provided by the function $f\equiv 1$.) It turns out that there are measures for which this range is in fact realized, with examples provided by Chen \cite{chen2}, Shmerkin and Suomala \cite{shmerkin-suomala2014}, and Chen and Seeger \cite{chen-seeger2015}. In particular, Chen and Seeger \cite{chen-seeger2015} proved that for $d\geq 1$ and $\alpha=d/k$, where $k\in\NN$, there are measures supported on a set of Hausdorff dimension $\alpha$, obeying (\ref{A}) and (\ref{B}) with $\beta=\alpha$, for which (\ref{restriction 1}) holds for all $p\geq 2d/\alpha$. The proofs are based on regularity of convolutions: assuming that $\alpha=d/k$, the key intermediate step is to prove that the $k$-fold self-convolution $\mu*\dots*\mu$ is absolutely continuous. This method, however,  does not yield optimal exponents when $\alpha\neq d/k$ with integer $k$.

Our main result is as follows. 

\begin{theorem}\label{main-thm}
Let $d\in\NN$ and $0<\alpha<d$. Then there exists a probability measure supported on a subset of $[0,1]^d$ of Hausdorff dimension $\alpha$ such that:
\begin{itemize}
\item for every $0<\gamma<\alpha$, there is a constant $C_1(\gamma)$ such that
\begin{equation}\label{main-e1}
\mu(B(x,r))\leq C_1(\gamma) r^\gamma \ \ \ \forall x\in\RR^d,\ r>0
\end{equation}
\item for every $\beta<\min(\alpha/2,1)$, there is a constant $C_2(\beta)>0$ such that
\begin{equation}\label{main-e2}
|\widehat{\mu}(\xi)| \leq C_2(\beta) (1+|\xi|)^{-\beta}\ \ \ \forall \xi\in\RR^d,
\end{equation}
\item for every $p>2d/\alpha$, we have the estimate
\begin{align}\label{restriction}
\| \widehat{g d\mu} \|_{p} \leq C_3(p) \| g \|_{L^2(\mu)} \quad \forall g \in L^2(\mu).
\end{align}
\end{itemize}
\end{theorem}

This complements the results of \cite{chen2}, \cite{chen-seeger2015}, \cite{shmerkin-suomala2014}, and provides a matching (except for the endpoint) result for all dimensions $0<\alpha<d$ that are not of the form $\alpha=d/k$.

The first main ingredient of our construction is Bourgain's theorem on $\Lambda(p)$ sets \cite{bourg-lambdaP} (see also Talagrand \cite{talagrand}). In its full generality, Bourgain's theorem applies to general bounded orthogonal systems of functions.
We state it here in the specific case of exponential functions on the unit cube in $\RR^d$. 
This provides an optimal restriction estimate on each single scale in the Cantor construction. 

\begin{theorem}\label{bourgain-thm} (Bourgain \cite{bourg-lambdaP})
Let $p>2$.
For every $N\in\NN$ sufficiently large, there is a set $S=S_N\subset \{0, 1,\dots, N-1\}^d$ of size $t\geq c_0 N^{2d/p}$ such that
\begin{equation} \label{lambda-p}
\|\sum_{a\in S} c_a e^{2\pi ia\cdot x}\|_{L^{p}[0,1]^d}\leq C(p) (\sum_{a\in S}|c_a|^2)^{1/2}
\end{equation}
with the constants $c_0$ and  $C(p)$ independent of $N$. (The set $S$ is called a $\Lambda(p)$-set.)
\end{theorem}

To pass from here to restriction estimates for multiscale Cantor sets, we use the decoupling techniques of Bourgain and Demeter \cite{BD2015}, \cite{BD-expo}. This produces 
localized restriction estimates of the form
\begin{equation}\label{r-e1}
\| \widehat{f} \|_{L^2(\mu)} \leq C_\epsilon R^\epsilon \|f\|_{L^{p'}([-R,R]^d)},
\end{equation}
or equivalently by duality,
\begin{equation}\label{r-e2}
\| \widehat{g d\mu} \|_{L^{p}([-R,R]^d)} \leq C_\epsilon R^\epsilon \|g\|_{L^{2}(\mu)}
\end{equation}
for all $\epsilon>0$, with constants independent of $R$. The $R^\epsilon$ factors account for the fact that we lose a constant factor at each step of the iteration. We will try to minimize these losses by applying Bourgain's theorem to an increasing sequence of values of $N$, but we will not be able to avoid them completely.

Finally, we use a variant of Tao's epsilon removal lemma \cite{tao-BR1999} to deduce the global restriction estimate (\ref{restriction}) from (\ref{r-e1}). This removes the $R^\epsilon$ factors, but at the cost of losing the endpoint exponent $p=2d/\alpha$. 
It is not clear whether the endpoint estimate can be obtained with our current methods.

Our proof of the localized restriction estimate (\ref{r-e1}) is fully deterministic. However, the epsilon removal lemma requires a pointwise Fourier decay estimate for $\mu$. Randomizing our construction enables us to prove 
the estimate (\ref{main-e2}) via an argument borrowed from \cite{LP}, \cite{shmerkin-suomala2014}. This proves the Fourier decay part of Theorem \ref{main-thm}, and is also sufficient to complete the epsilon removal argument. 

If $d\leq 2$, or if $d\geq 2$ and $\alpha\geq d-2$, the Cantor set supporting $\mu$ in Theorem \ref{main-thm} is a Salem set (i.e. its Fourier
dimension is equal to its Hausdorff dimension). The condition $\alpha\geq d-2$ is necessary for this type of constructions to produce a Salem set, for the same reasons as in \cite{shmerkin-suomala2014}. We note, however, that our proof of (\ref{restriction}) with $p>2d/\alpha$ does not require optimal Fourier decay and that the estimate (\ref{main-e2}) for {\em any} $\beta>0$ would suffice.


\section{The decoupling machinery}

We will use the decoupling machinery developed by Bourgain and Demeter \cite{BD2015}, \cite{BD-expo}. In this paper, we will follow the conventions of \cite{BD-expo},
with the surface measure on a paraboloid replaced by the natural measure on a Cantor set. 

We use $X \lesssim Y$ to say that $X \leq C Y$ for some constant $C>0$, and $X \approx Y$ to say that $X \lesssim Y$ and $X \gtrsim Y$. The constants such as $C,C'$, etc. and the implicit constants in $\lesssim$ may change from line to line, and may depend on $d$ and $p$, but are independent of variables or parameters such as $x,N,R, j, \ell$.
For quantities that depend on parameters such as $\epsilon$, we will write $X(\epsilon) \lesssim_\epsilon Y(\epsilon) $ as shortcut for ``for every $\epsilon>0$ there is a constant $C_\epsilon>0$ such that $X(\epsilon) \leq C_\epsilon Y(\epsilon)$."

We write $[N]=\{0,1,\dots,N-1\}$ and $B(x,r)=\{y\in\RR^d:\ |x-y|\leq r\}$. We use $|\cdot |$ to use the Euclidean ($\ell^2$) norm of a vector in $\RR^d$, the cardinality of a finite set, or the $d$-dimensional Lebesgue measure of a subset of $\RR^d$, depending on the context. Occasionally, we will also use the $\ell^\infty$ norm on $\RR^d$: if $x=(x_1,\dots,x_d)\in \RR^d$, we
write $|x|_\infty=\max(|x_1|,\dots,|x_d|)$. We will also sometimes use $\mathcal{F}$ for the Fourier transform, so that $\mathcal{F}f =\widehat{f}$.

%
%
%

Following \cite{BD-expo}, we will use cube-adjusted weights. An {\it $R$-cube} will be a $d$-dimensional cube of side length $R$, with all sides parallel to coordinate hyperplanes. Unless stated otherwise, we will assume $R$-cubes to be closed. If $I$ is an $R$-cube centered at $c$, we define
$$
w_I(x)=\left(1+\frac{|x-c|}{R}\right)^{-100}
$$
and 
$$
\|F\|_{L^p_\sharp (w_I)}=\left( \frac{1}{|I|} \int |F|^p w_I\right)^{1/p}.
$$
If $\eta:\RR\to[0,\infty)$ is a function (usually Schwartz), and $I$ is as above, we will write
$$
\eta_I(x)=\eta\left(\frac{x-c}{R}\right).
$$
If $g:\RR\to\CC$ is a function, $I$ is an interval, and $\sigma$ is a measure (which will usually be clear from context), we will write
$$
E_Ig =\mathcal{F}^{-1}(\one_I g d\sigma).
$$
We will use the following tools from Bourgain-Demeter, which we restate here in a version adapted to our setting.

\begin{lemma} \label{bd-lemma4.2}
{\bf (Reverse H\"older inequality, \cite[Corollary 4.2]{BD-expo}).} Let $1\leq p\leq q$. If $I$ is a $1/R$-cube and $J$ is an $R$-cube, then  
\begin{equation}\label{reverse-Holder}
\|E_Ig\|_{L^q_\sharp (w_J)} \lesssim \|E_Ig\|_{L^p_\sharp (w_J)} 
\end{equation}
with the implicit constant independent of $R,I,J,g$.
\end{lemma}

\begin{lemma} \label{bd-prop6.1}
{\bf ($L^2$ decoupling, \cite[Proposition 6.1]{BD-expo}).} Let $I$ be a $k/R$-cube for some $k\in\NN$, and let $I=I_1\cup \dots \cup I_k$ be a tiling of $I$ by $1/R$-cubes disjoint except for their boundaries. Then for any $R$-cube $J$ we have 
\begin{equation}\label{l2-decoupling}
\|E_Ig\|^2_{L^2 (w_J)} \lesssim \sum_j \|E_{I_j}g\|^2_{L^2 (w_J)} .
\end{equation}
\end{lemma}

\begin{lemma} \label{alt-locally-constant}
{(\bf Band-limited functions are locally constant, cf. \cite[\S2.2]{BG})}
There is a non-negative function $\eta\in L^1(\RR^d)$ such that the following holds. For every $R>0$, and every integrable function $h:\RR^d\to\CC$ supported on a $1/R$-cube $I$, there is a function $H:\RR\to[0,\infty)$ such that:
\begin{itemize}
\item $H$ is constant on each semi-closed $R$-cube $J_\nu:=R\nu+[0,R)^d$, $\nu\in\ZZ^d$,
\item $|\widehat{h}(x)|\leq H(x)\leq (|\widehat{h}|*\eta_R)(x)$ for all $x\in\RR$, where $\eta_R(y)=\frac{1}{R^d}\eta(\frac{y}{R})$. In particular, 
\begin{equation}\label{HNotTooLarge}
\|H\|_{L^1(\RR^d)}\leq \|\eta\|_{L^1(\RR^d)} \|\widehat{h}\|_{L^1(\RR^d)}.
\end{equation}
\end{itemize}
\end{lemma}

\begin{proof}
Replacing $h$ by $h(\cdot -c)$ and $\widehat{h}(x)$ by $e^{-2\pi i c\cdot x}\widehat{h}(x)$ if necessary, we may assume that
$I=[0,\frac{1}{R}]^d$. 
Let $\chi$ be a non-negative Schwartz function such that $\chi\equiv  1$ on $[0,1]^d$ and that $\widehat{\chi}(x)$ is non-negative, radially symmetric and decreasing in $|x|$. Then $\chi(R\cdot)\equiv 1$ on
$I$, and $\widehat{\chi(R\cdot})= \frac{1}{R^d}\widehat{\chi}(\frac{\cdot}{R})$. Define
$$
\eta(x):=\sup_{|y-x|_\infty \leq 1} \widehat{\chi(y)}
$$
and
$$
H(x):=\sup\{|\widehat{h}(y)|:\ x,y\hbox{ belong to the same }J_\nu\}.
$$
Clearly, $\eta$ is integrable and $H$ is constant on each $J_\nu$. We have $|\widehat{h}(x)|\leq H(x)$ by definition. To prove the second inequality, we note that $h=h\chi(R\cdot)$, so that $\widehat{h}=\widehat{h}*\widehat{\chi(R\cdot)}$. Suppose that $x\in J_\nu$ for some $\nu\in\ZZ^d$, then for each $y\in J_\nu$ we have
$$
|\widehat{h}(y)|\leq \int |\widehat{h}(z)|  \,\frac{1}{R^d}\, \widehat{\chi}(\frac{y-z}{R})dz
$$
Since $|x-y|_\infty \leq R$, we have $|\frac{y-z}{R}-\frac{x-z}{R}|_\infty=|\frac{y-x}{R}|_\infty \leq 1$, so that by the definition of $\eta$ we have
$\eta(\frac{x-z}{R})\geq \widehat{\chi}(\frac{y-z}{R})$. Hence
$$
|\widehat{h}(y)|\leq \int |\widehat{h}(z)|  \,\frac{1}{R^d}\, \eta(\frac{x-z}{R})dz = (|\widehat{h}|*\eta_R)(x),
$$
and the desired inequality follows upon taking the supremum over $y\in J_\nu$. Finally, by Fubini's theorem and rescaling we have
$$
\|H\|_{L^1(\RR^d)}
\leq \|\widehat{h}\|_{L^1(\RR^d)} \|\eta_R\|_{L^1(\RR^d)}
= \|\widehat{h}\|_{L^1(\RR^d)}\|\eta\|_{L^1(\RR^d)} .
$$
\end{proof}

\begin{corollary}\label{small intervals}
For every $R>0$, $M\in\NN$, every integrable function $h:\RR^d\to\CC$ supported on an $(MR)^{-1}$-cube $I$, and every $R$-cube $J$,
we have
\begin{equation}\label{small intervals eq}
\|\widehat{h}\|_{L^1(w_J)} \lesssim \frac{1}{M^d} \|\widehat{h}\|_{L^1(\RR^d)} .
\end{equation}
\end{corollary}

\begin{proof}
Let $L_\nu=MR\nu + [0,MR)^d$ for $\nu\in\ZZ^d$. Let $H$ be the function provided by Lemma \ref{alt-locally-constant} with $R$ replaced by $MR$, so that on each $L_\nu$ we have $H(x)\equiv H_\nu$ for some constant $H_\nu\geq 0$. Then
\begin{align*}
\|\widehat{h}\|_{L^1(w_J)} & \leq \int H(x) w_J(x) dx = \sum_\nu  H_\nu \int_{L_\nu} w_J(x) dx
\\
&\leq \sum_\nu H_\nu \int_{\RR} w_J(x) dx \\
&= \sum_\nu H_\nu R^d \int_{\RR} w_{[0,1]^d}(x) dx.
\end{align*}
Let $C_1:= \int_{\RR^d} w_{[0,1]^d}(x) dx$, then
\begin{align*}
\|\widehat{h}\|_{L^1(w_J)} 
&\leq C_1 \sum_\nu H_\nu R^d = \frac{C_1}{M^d}\sum_\nu H_\nu (MR)^d\\
&=\frac{C_1}{M^d}\int_{\RR} H(x) dx \\
&\lesssim \frac{1}{M^d}  \|\widehat{h}\|_{L^1(\RR^d)},
\end{align*}
where at the last step we used (\ref{HNotTooLarge}).
\end{proof}

\section{Single-scale decoupling}

We begin with a single-scale decoupling inequality for Cantor sets with $\Lambda(p)$ alphabets. 
We will need the following ``continuous" version of Theorem \ref{bourgain-thm}.

\begin{lemma}\label{lemma-conversion}
Let $p>2$, and let $S\subset[N]^d$ be as in Theorem \ref{bourgain-thm}. Then for all $h$ supported on $E:=S+[0,1]^d$ we have the inequality
\begin{equation}\label{dec-e4}
\|\widehat{h}\|_{L^p([0,1]^d)}\lesssim C(p) \|h\|_{L^2(\RR^d)}
\end{equation}
\end{lemma}

\begin{proof}
We have
\begin{align*}
\Big\| \sum_{a\in S} c_a e^{2\pi i a\cdot x}\Big\| _{L^p([0,1]^d)}
&= \sup_{ \|f\|_{L^{p'}([0,1]^d)}=1} \Big\langle f, \sum_{a\in S} c_a e^{2\pi i a\cdot x} \Big\rangle
\\
&= \sup_{ \|f\|_{L^{p'}([0,1]^d)}=1} \Big\langle \widehat{f}, \sum_{a\in S} c_a \delta_a \Big\rangle
\\
&= \sup_{ \|f\|_{L^{p'}([0,1]^d)}=1} \sum_{a\in S} c_a \widehat{f}(a)
\end{align*}
By (\ref{lambda-p}), it follows that
\[
\sup_{\|c_a\|_{\ell^2(S)}=1} \sup_{ \|f\|_{L^{p'}([0,1]^d)}=1} \sum_{a\in S} c_a \widehat{f}(a) \leq C(p),
\]
so that
\[
\|\widehat{f}(a)\|_{\ell^2(S)} \leq C(p) \|f\|_{L^{p'}([0,1]^d)}
\]
Similarly, for any translate $S+z$ of $S$ we have
\[
\|\widehat{f}(a)\|_{\ell^2(S+z)} \leq C(p) \|f\|_{L^{p'}([0,1]^d)}
\]
Integrating in $z\in[0,1]^d$, we get
\begin{equation}\label{dec-e12}
\|\widehat{f}\|^2_{L^2(E)} = \int_{[0,1]^d} \|\widehat{f}(a)\|_{\ell^2(S+z)} dz \leq C(p)^2 \|f\|^2_{L^{p'}([0,1]^d)}.
\end{equation}
Arguing again by duality, we have
\begin{align*}
\|\widehat{f}\|_{L^2(E)}
&= \sup_{ \|g\|_{L^{2}(E)}=1} \int g\one_{E} \widehat{f} dx\\
&= \sup_{ \|g\|_{L^{2}(E)}=1} \int \mathcal{F}(g\one_{E})\, f  dx
\end{align*}
Using (\ref{dec-e12}), and taking the supremum over $f$ with $\|f\|_{L^{p'}([0,1]^d)}\leq 1$, 
we get 
\[
\|\mathcal{F}(g\one_{E})\|_{L^p([0,1]^d)} \lesssim C(p)\|g\|_{L^{2}(E)},
\]
which is (\ref{dec-e4}) with $h=g\one_{E}$.
\end{proof}

We note that the conclusion of Lemma \ref{lemma-conversion} remains true if we assume that $h$ is supported on $S+[-1/2, 3/2]^d$ instead of $E$. This is proved by writing $h$ as a sum of $2^d$ functions supported on translates of $E$ and applying Lemma \ref{lemma-conversion} to each of them.

We can now prove our first decoupling inequality.

\begin{lemma}\label{discrete to continuous} 
Let $S\subset [N]^d$ be a $\Lambda(p)$-set as in Theorem \ref{bourgain-thm}, and let $E=S+[0,1]^d$. Let $f:\RR\to\CC$ be a function such that $g:=\widehat{f}=$ is supported on 
$E$. For each $a\in S$, let $g_a=g\one_{a+[0,1]^d}$, and define $f_a$ via $\widehat{f_a}=g_a$. Then
\begin{equation}\label{decoupling}
\|f\|^2_{L^p(w_I)}\lesssim C(p)^2 \sum_{a\in S} \|f_a\|^2_{L^p(w_I)}
\end{equation}
for any 1-cube $I$.
\end{lemma}

\begin{proof} 
We first rewrite the right-hand side of (\ref{decoupling}) using Lemmas \ref{bd-prop6.1} and \ref{bd-lemma4.2} with $\sigma$ equal to the Lebesgue measure. We have
$$
E_{a+[0,1]^d}g=\widehat{g_a}=f_a,\ \ \ E_{[0,N]^d}g = \widehat{g} = f,
$$
so that
$$
\|f\|^2_{L^2(w_I)}\approx  \sum_{a\in S} \|f_a\|^2_{L^2(w_I)}
\approx \sum_{a\in S}  \|f_a\|^2_{L^p(w_I)}.
$$
Therefore to prove (\ref{decoupling}), it suffices to prove that
\begin{equation}\label{dec-e2}
\|f\|^2_{L^p(w_I)}\lesssim C(p)^2  \|f\|^2_{L^2(w_I)}
\end{equation}
Let $\eta$ be a nonnegative Schwartz function such that $\eta(x)=\eta(-x)$, $\eta\geq 1$ on $[-1,1]^d$ and supp$\widehat{\sqrt{\eta}}\subset[-\frac{1}{2},\frac{1}{2}]^d$. We will prove that 
\begin{equation}\label{dec-e5}
\|f\|^2_{L^p(I)}\lesssim C(p)^2  \|f\|^2_{L^2(\eta_I)}
\end{equation}
for every 1-cube $I$.
By a covering argument \cite[Lemma 4.1]{BD-expo}, this implies (\ref{dec-e2}).

We may assume that $I=[0,1]^d$. (If $I=z+[0,1]^d]$, we may replace $f$ by $f_z=f(\cdot +z)$ and observe that 
$\widehat{f_z}(\xi)=e^{2\pi i z\cdot \xi} \widehat{f}(\xi)$ is again supported in $E$.)
Let $h=g* (\sqrt{\eta})\,\check{ }$, so that $\widehat{h}=f\sqrt{\eta}$ and $h$ is supported on $S+[-1/2,3/2]^d$. 
Since $\sqrt{\eta}\geq 1$ on $[0,1]^d$, 
we have
$\|f\|_{L^p([0,1]^d)} \leq  \|f\sqrt{\eta} \|_{L^p([0,1]^d)}$.
By Lemma \ref{lemma-conversion} and the remark after its proof applied to $h$,
\[
\|f\sqrt{\eta} \|_{L^p([0,1]^d)}  \lesssim C(p) \|h\|_{L^2(\RR)}
= \|f\sqrt{\eta} \|_{L^2(\RR)}=\|f\|_{L^2(\eta)}
\]
as claimed.

\end{proof}

\section{The Cantor set construction}\label{cantor set}

Our proof of Theorem \ref{main-thm} is based on the construction of a ``multiscale $\Lambda(p)$" Cantor set of dimension $\alpha$. Let $\alpha\in(0,d)$, $p=2d/\alpha$, and let $\{n_j\}_{j\in\NN}$ be a sequence of positive integers. For the construction of the measure $\mu$ in Theorem \ref{main-thm}, we will assume the following conditions on $n_j$:
\begin{equation}\label{baby steps}
n_1\leq n_2\leq\dots,\ \ n_k\to\infty,\ \ 
\end{equation}
\begin{equation}\label{baby epsilons}
\forall\epsilon>0 \ \exists C_\epsilon >0 \ \forall k\in\NN\ \  n_{k+1} \leq C_\epsilon (n_1\dots n_k)^\epsilon.
\end{equation}
However, large parts of our proof work under weaker assumptions. In particular, our localized restriction estimate in Lemma
\ref{random cantor measure} continues to hold if $n_j=n$ for all $j$. We also note here that in order for (\ref{baby epsilons}) to hold, it is enough to assume that (\ref{baby steps}) holds and that $n_j$ grow slowly enough, for example
\begin{equation}\label{slow growth}
\frac{n_{j+1}}{n_j}\leq \frac{j+1}{j}
\end{equation}
will suffice.

For each $j\in\NN$, let 
$$
\Sigma_j=\Sigma_j(n_j,t_j,c_0,C(p))=\{S\subset[n_j]^d: \ |S|=t_j\hbox{ and (\ref{lambda-p}) holds with }N=n_j\}
$$
By Theorem \ref{bourgain-thm}, there are $c_0,C(p)>0$ (independent of $j$) and $t_j$ with $t_j\geq c_0 n_j^{2d/p}$
such that $\Sigma_j$ is non-empty for all $j$. Henceforth, we fix these values of $c_0$, $C(p)$ and $t_j$. By the well known 
upper bounds on the size of $\Lambda(p)$ sets (see \cite{bourg-lambdaP}), we must in fact have 
\begin{equation}\label{large t}
c_0 n_j^{2d/p}\leq t_j\leq c_1 n_j^{2d/p}
\end{equation}
for some constant $c_1$ independent of $j$.

Let $N_k=n_1\dots n_k$ and $T_k=t_1\dots t_k$. We construct a Cantor set $E$ of Hausdorff dimension $\alpha$ as follows. Define
$$A_1=N_1^{-1}S_1, \ \ E_1= A_1+[0,N_1^{-1}]^d,$$
for some $S_1\in\Sigma_1$.
For every $a\in A_1$, choose a $\Lambda(p)$ set $S_{2,a}\in\Sigma_2$ with $|S_{2,a}|=t_2$, and let
$$
A_{2,a}=a+N_2^{-1}S_{2,a},\ \ 
A_2=\bigcup_{a\in A_1} A_{2,a},\ \ E_2=A_2+[0,N_2^{-1}]^d.
$$
We continue by induction. Let $k\geq 2$, and suppose that we have defined the sets $A_j$ and $E_j$, $j=1,2,\dots,k$.
For every $a\in A_k$, choose $S_{k+1,a}\in\Sigma_{k+1}$ with $|S_{k+1,a}|=t_{k+1}$, and let
$$
A_{k+1,a}=a+N_{k+1}^{-1}S_{k+1,a},\ \ 
A_{k+1}=\bigcup_{a\in A_k} A_{k+1,a},\ \ E_{k+1}=A_{k+1}+[0,N_{k+1}^{-1}]^d.
$$
This produces a sequence of sets $[0,1]^d \supset E_1\supset E_2\supset E_3\supset\dots$, where each $E_j$ consists of 
$T_j$ cubes of side length $N_j^{-1}$. For each $j$, let
$$
\mu_j=\frac{1}{|E_j|} \one_{E_j}.$$
We will identify the functions $\mu_j$ with the absolutely continuous measures $\mu_j\,dx$. It is easy to see that $\mu_j$ converge weakly as $j\to\infty$ to a probability measure $\mu$ supported on the Cantor set $E_\infty:=\bigcap_{j=1}^\infty E_j$.
We note that for each $N_j^{-1}$-cube $\tau$ of $E_j$, and for all $\ell >j$, we have $\mu_j(\tau)=\mu_\ell(\tau)=\mu(\tau)=T_j^{-1}$.

For the time being, the specific choice of the sets $A_{k,a}$ does not matter, as long as they are $\Lambda(p)$-sets of the prescribed cardinality. Our multiscale decoupling inequality in Proposition \ref{multidecoupling} and the localized restriction estimate in Corollary \ref{localized-restriction} do not require
any additional conditions. However, additional randomization of these choices will become important later in proving our global restriction estimate.

\begin{lemma}\label{dimension}
Assume that (\ref{baby steps}) and (\ref{baby epsilons}) hold.
Then the set $E_\infty$ has Hausdorff dimension $\alpha$.
Moreover, for every $0\leq \gamma<\alpha$ there is a constant $C_1(\gamma)$ such that
\begin{equation}\label{main-e1b}
\mu(B(x,r))\leq C_1(\gamma) r^\gamma \ \ \ \forall x\in\RR^d,\ r>0.
\end{equation}
\end{lemma}

\begin{proof}
We first note that (\ref{baby steps}), (\ref{baby epsilons}) and (\ref{large t}) imply that
\begin{equation}\label{t vs n}
N_{j+1}^{\alpha-2\epsilon} \lesssim_\epsilon N_{j}^{\alpha-\epsilon} \lesssim_\epsilon 
T_j \lesssim_\epsilon 
N_{j}^{\alpha+\epsilon} \lesssim_\epsilon 
N_{j-1}^{\alpha+2\epsilon}.
\end{equation}
Indeed, from (\ref{large t}) we have $c_0^j N_j^\alpha \leq T_j \leq c_1^j N_j^\alpha$, which implies  
$N_{j}^{\alpha-\epsilon} \lesssim_\epsilon T_j \lesssim_\epsilon N_{j}^{\alpha+\epsilon} $ by (\ref{baby steps}). The remaining two inequalities in (\ref{t vs n}) follow from (\ref{baby epsilons}).

We first prove (\ref{main-e1b}). If $r>N_1^{-1}$, then (\ref{main-e1b}) holds trivially with $C_1=N_1^\gamma$.  Suppose now that $N_{j+1}^{-1}<r\leq N_j^{-1}$ for some $j\geq 1$. Then any ball $B(x,r)$ intersects at most a bounded number of the $N_j^{-1}$-cubes of $E_j$, so that $\mu(B(x,R))\lesssim T_j^{-1}\lesssim_\gamma N_{j+1}^{-\gamma}\leq r^{\gamma}$ for all $\gamma<\alpha$ (the second inequality in the sequence follows from (\ref{t vs n})).

To prove the dimension statement, we only need to show that $E_\infty$ has Hausdorff dimension at most $\alpha$, since the lower bound is provided by (\ref{main-e1b}). To this end, it suffices to check that for every $\epsilon>0$, and for all $r>0$, the set $E_\infty$ can be covered by $C_\epsilon r^{-\alpha-\epsilon}$ balls of radius $r$. Again, it suffices to consider $r>N_1^{-1}$. Suppose that $N_{j+1}^{-1}<r\leq N_j^{-1}$, then $E_\infty\subset E_{j+1}$ can be covered by $\lesssim T_{j+1}$ balls of radius $N_{j+1}^{-1}$, hence also of radius $r$. Since $T_{j+1}\lesssim_\epsilon N_{j}^{\alpha+\epsilon}\leq r^{-\alpha-\epsilon}$, the desired bound follows.
\end{proof}


\section{Multiscale decoupling}\label{multidec-section}

Our goal in this section is to derive the following multiscale decoupling inequality for finite iterations of Cantor sets. For $a\in A_k$, let $\tau_{k,a}=a+[0,N_k^{-1}]^d$. If $f:\RR^d\to\CC$ is a function, we define $f_{k,a}$ via $\widehat{f_{k,a}}= \one_{\tau_{k,a}} \widehat{f}$.

\begin{proposition}\label{multidecoupling}
There is a constant $C_0(p)$ (independent of $k$) such that for any $N_k$-cube $J$, and for any function $f$ with supp$\widehat{f}\subseteq E_k$, we have 
\begin{equation}\label{multi-e}
\left( \sum_{I\in\mathcal{I}} \|f\|^p_{L^p(w_I)}\right)^{1/p} 
\leq C_0(p)^{k} \left(\sum_{a\in A_k} \|f_{k,a}\|_{L^p(w_J)}^2\right)^{1/2},
\end{equation}
where $J=\bigcup_{I\in\mathcal{I}}I$ is a tiling of $J$ by 1-cubes.
\end{proposition}

\begin{proof}
The idea is to iterate Lemma \ref{discrete to continuous}. 
Applying it to the set $N_1\cdot E_1$ and a rescaling of $f$ by $N_1$, 
we see that there is a constant $C_1(p)$ such that for any function  $f$ with supp $\widehat{f}\subseteq E_1$, and for any $N_1$-cube $J$, we have
\begin{equation} \label{best ball}
\|f\|^2_{L^p(w_{J})}\leq C_1(p)^2 \sum_{a\in A_1} \|f_{1,a}\|_{L^p(w_{J})}^2 .
\end{equation}
Similarly, applying Lemma~\ref{discrete to continuous} to a rescaling of $f_{j,a}$ by $N_{j+1}$ for each $a\in  A_j$, we see that for any $N_{j+1}$-cube $J$ and for any function  $f$ with supp $\widehat{f}\subseteq E_{j+1}$
we have
\begin{equation} \label{best big ball}
\|f_{j,a}\|^2_{L^p(w_{J})}\leq C_1(p)^2 \sum_{b\in A_{j+1,a}} \|f_{j+1,b}\|_{L^p(w_{J})}^2 
\end{equation}
with the same constant $C_1(p)$.

To connect the steps of the iteration, we will need a simple lemma on mixed norms.

\begin{lemma}\label{mixed norms}
Let $\{c_{ij}\}$ be a double-indexed sequence (finite or infinite) with $c_{ij}\geq 0$. Then for $p>2$,
\begin{equation}\label{md-e1}
\sum_i \Big( \sum_j c_{ij}^2\Big)^{p/2} 
\leq \Big( \sum_j \Big( \sum_i c_{ij}^p\Big)^{2/p} \Big)^{p/2}.
\end{equation}
\end{lemma}

\begin{proof}
Let $F_j(i)=c_{ij}^2$, and $G(i)=\sum_j F_j(i)=\sum_{j}c_{ij}^2$, so that 
$$
\|G\|_{p/2} = \Big( \sum_i \Big( \sum_j c_{ij}^2\Big)^{p/2} \Big)^{2/p}.
$$
On the other hand, by Minkowski's inequality
$$
\|G\|_{p/2}\leq \sum_j \|F_j(i)\|_{p/2} = \sum_j \Big( \sum_i c_{ij}^p\Big)^{2/p},
$$
and the lemma follows.
\end{proof}

We will prove (\ref{multi-e}) by induction in $k$. For an $m$-cube $J$ with $m\in\NN$, let $J=\bigcup_{I\in \mathcal{I}(J)} I$ be a tiling of $J$ by 1-cubes. Let $C_2$ be a constant such that
\begin{equation}\label{weight-compare}
\sum_{I\in \mathcal{I}(J)} w_I \leq C_2 w_J.
\end{equation}
It is easy to see that such a constant exists and can be chosen independently of $|J|$. We will prove that (\ref{multi-e}) holds with $C_0(p)=C_1(p)C_2^{1/p}$.

To start the induction, let $J$ be an $N_1$-cube, then by (\ref{weight-compare}) and (\ref{best ball}),
\begin{align*}
\sum_{I\in\mathcal{I}(J)} \|f\|^p_{L^p(w_I)}
&\leq C_2 \|f\|^p_{L^p(w_{J})}\\
&\leq C_1(p)^p C_2 \Big( \sum_{a\in A_1} \|f_{1,a}\|_{L^p(w_{J})}^2 \Big)^{p/2}
\end{align*}
This is (\ref{multi-e}) for $k=1$. Suppose now that we have proved 
(\ref{multi-e}) for $k=j$. Let $J$ be an $N_{j+1}$-cube, and let $J=\bigcup_{L\in \mathcal{L}} I$ be a tiling of $J$ by $N_j$-cubes. Then
\begin{align*}
\sum_{I\in\mathcal{I}(J)} \|f\|^p_{L^p(w_I)}
&= \sum_{L\in\mathcal{L}} \sum_{I\in\mathcal{I}(L)} \|f\|^p_{L^p(w_I)}\\
&\leq C_0(p)^{jp} \sum_{L\in\mathcal{L}} \Big( \sum_{a\in A_j} \|f_{j,a}\|_{L^p(w_{L})}^2 \Big)^{p/2}
\end{align*}
by our inductive assumption. 
Using Lemma \ref{mixed norms}, a rescaling of (\ref{weight-compare}), and (\ref{best big ball}), we see that
\begin{align*}
\sum_{I\in\mathcal{I}(J)} \|f\|^p_{L^p(w_I)}
&\leq C_0(p)^{jp} \left[ \sum_{a\in A_j} \left( \sum_{L\in\mathcal{L}} \|f_{j,a}\|_{L^p(w_{L})}^p \right)^{2/p} \right]^{p/2}
\\
&\leq C_0(p)^{jp} C_2 \left[ \sum_{a\in A_j}  \|f_{j,a}\|_{L^p(w_{J})}^2  \right]^{p/2}
\\
&\leq C_0(p)^{jp} C_2 C_1(p) \left[ \sum_{a\in A_{j+1}}  \|f_{j+1,a}\|_{L^p(w_{J})}^2  \right]^{p/2}.
\end{align*}
This ends the inductive step and the proof of the proposition.

\end{proof} 


\section{From decoupling to localized restriction}

\begin{lemma} \label{random cantor measure}
Let $E_k$ and $\mu_k$ be as in Section \ref{multidec-section}. 
Let $J$ be an $N_k$-cube. Then
for all $g\in L^2(d\mu)$, we have
$$\|\widehat{g d\mu}\|_{L^p(J)}\lesssim C_0(p)^k N_k^{d/p} \,T_k^{-1/2} \|g\|_{L^2(d\mu)}$$
with the implicit constants independent of $k$.
\end{lemma}

\begin{proof}
It suffices to prove that for all $\ell>k$, and for all $g\in L^2(\RR)$ supported on $E_\ell$, we have
$$\|\widehat{g d\mu_{\ell}}\|_{L^p(J)}\lesssim C_0(p)^k N_k^{d/p} \,T_k^{-1/2} \|g\|_{L^2(d\mu_{\ell})}$$
with the implicit constants independent of $k$ and $\ell$. The claim then follows by taking the limit $\ell\to\infty$.

We continue to use the Cantor set notation from Sections \ref{cantor set} and \ref{multidec-section}. For $a\in A_j$, let
$g_{j,a}=\one_{a+[0,+N_j^{-1}]^d}g$. 
By Proposition \ref{multidecoupling} and Lemma \ref{bd-lemma4.2}, we have
\begin{align*}
\|\widehat{gd\mu_{\ell}}\|_{L^p(J)} & 
\lesssim C_0(p)^k \Big(\sum_{a\in A_k} \|\widehat{g_{k,a}d\mu_{\ell}}\|_{L^p(w_J)}^2\Big)^{1/2}\\
& \approx  C_0(p)^k N_k^{\frac{d}{p}-\frac{d}{2}} \Big(\sum_{a\in A_k} \|\widehat{g_{k,a} d\mu_{\ell}} \|_{L^2(w_J)}^2\Big)^{1/2}
\end{align*}
For each $a\in A_k$, let $B_{\ell,a}$ be the set of $\ell$-th level ``descendants" of $a$ (more precisely, $B_{\ell,a}=\{b\in A_\ell:\ b+[0,N_\ell^{-1}]^d\subset a+[0,N_k^{-1}]^d\}$. Note that $|B_{\ell,a}|=T_\ell/T_k$. By Cauchy-Schwartz,
$$ 
\|\widehat{g_{k,a} d\mu_{\ell}} \|_{L^2(w_J)}
\leq \sum_{b\in B_{\ell,a}} \|\widehat{g_{\ell,b} d\mu_{\ell}} \|_{L^2(w_J)}
\leq \left(\frac{T_\ell}{T_k}\right)^{1/2}
 \Big(\sum_{b\in B_{\ell,a}} \|\widehat{g_{\ell,b} d\mu_{\ell}} \|_{L^2(w_J)}^2\Big)^{1/2}
$$
so that
\begin{align*}
\|\widehat{gd\mu_{\ell}}\|_{L^p(J)} 
&\lesssim C_0(p)^k N_k^{\frac{d}{p}-\frac{d}{2}}\left(\frac{T_\ell}{T_k}\right)^{1/2}
\Big(\sum_{b\in A_\ell} \|\widehat{g_{\ell,b} d\mu_{\ell} } \|_{L^2(w_J)}^2\Big)^{1/2} \\
&\lesssim C_0(p)^k N_k^{\frac{d}{p}-\frac{d}{2}}\left(\frac{T_\ell}{T_k}\right)^{1/2}
\left(\frac{N_k}{N_\ell}\right)^{d/2}
\Big(\sum_{b\in A_\ell } \|\widehat{g_{\ell,b} d\mu_{\ell} } \|_{L^2(\RR)}^2\Big)^{1/2}.
\end{align*}
At the last step, we applied Corollary \ref{small intervals} to the functions $(g_{\ell,b} d\mu_\ell)(\cdot)* (g_{\ell,b} d\mu_\ell) (-\,\cdot)$ supported on $2N_\ell^{-1}$-cubes.

Since
\begin{align*}
\sum_{b\in A_\ell} \|\widehat{g_{\ell,b} d\mu_{\ell} } \|_{L^2(\RR)}^2
&=  \sum_{b\in A_\ell} \|{g_{\ell,b} d\mu_{\ell} } \|_{L^2(\RR)}^2\\
&= \|{g d\mu_{\ell} } \|_{L^2(\RR)}^2=  N_\ell^d T_\ell^{-1} \|g \|_{L^2(\mu_\ell)}^2
\end{align*}
we finally have
\begin{align*}
\|\widehat{gd\mu_{\ell}}\|_{L^p(J)} 
&\lesssim 
C_0(p)^k N_k^{\frac{d}{p}-\frac{1}{2}}\left(\frac{T_\ell}{T_k}\right)^{1/2}
\left(\frac{N_k}{N_\ell}\right)^{d/2} \left(\frac{N_\ell^d}{T_\ell}\right)^{1/2}
 \|g \|_{L^2(\mu_\ell)}\\
&=
C_0(p)^k N_k^{d/p}\, T_k^{-1/2} \|g\|_{L^2(d\mu_\ell)}
\end{align*}
as claimed.
\end{proof}

\begin{corollary}\label{localized-restriction} {\bf (Localized restriction estimate)}
Assume that  
(\ref{baby steps}) holds. Then for any $\epsilon>0$ we have the estimate
\begin{equation}\label{loc-e1}
\|\widehat{g d\mu}\|_{L^p(J)}\leq C_\epsilon R^\epsilon \|g\|_{L^2(d\mu)}.
\end{equation}
for all $R\geq n_1$ and for all $R$-cubes $J$. The constant $C_\epsilon$ depends on $\epsilon$, but not on $g$, $R$ or $J$.
Equivalently, for any $f$ supported in $J$, we have 
\begin{equation}\label{loc-e2}
\|\widehat{f}\|_{L^2(d\mu)} \leq C_\epsilon R^\epsilon \|f\|_{L^{p'}(J)}.
\end{equation}
\end{corollary}

\begin{proof}
Suppose that $N_k< R\leq N_{k+1}$, and let $J'$ be an $N_{k+1}$-cube containing $J$.
By Lemma \ref{random cantor measure} and (\ref{baby steps}), we have
\begin{align*}
\|\widehat{gd\mu}\|_{L^p(J)} 
\leq \|\widehat{gd\mu}\|_{L^p(J')} 
&\lesssim 
C_0(p)^{k+1} N_{k+1}^{d/p}\, T_{k+1}^{-1/2} \|g\|_{L^2(d\mu)}\\
&\lesssim 
C_0(p)^{k+1} N_{k+1}^{d/p}\, (c_0^{k+1} N_{k+1}^{2d/p})^{-1/2} \|g\|_{L^2(d\mu)}\\
&\lesssim (C_0(p)c_0^{-1/2})^{k+1} \|g\|_{L^2(d\mu)}\\
&\lesssim_\epsilon R^\epsilon \|g\|_{L^2(d\mu)}
\end{align*}
as claimed.
The second part (\ref{loc-e2}) follows by duality.
\end{proof}


\section{Global restriction estimate}

\begin{proposition}\label{global-restriction} 
Assume that $n_k$ obey (\ref{baby steps}) and (\ref{baby epsilons}).
Suppose furthermore that $\widehat{\mu}$ obeys a pointwise decay estimate
\begin{equation}\label{mu-decay}
|\widehat{\mu}(x)|\lesssim (1+|x|)^{-\beta}
\end{equation}
for some $\beta>0$.
Then for any $q>p$ we have the estimate
\begin{equation}\label{global-e1}
\|\widehat{g d\mu}\|_{L^q(\RR)}\lesssim \|g\|_{L^2(d\mu)}.
\end{equation}
The implicit constant depends on the measure $\mu$
and on $q$, but not on $g$.
Equivalently, 
\begin{equation}\label{global-e2}
\|\widehat{f}\|_{L^2(d\mu)} \lesssim \|f\|_{L^{q'}(\RR)}.
\end{equation}
\end{proposition}

To prove this, we adapt Tao's epsilon-removal argument, see \cite[Theorem 1.2]{tao-BR1999}. It suffices to prove Lemma \ref{epsilon-remove} below; once this is done, the proof of the proposition is completed exactly as in \cite{tao-BR1999}, with Lemma 
\ref{epsilon-remove} replacing Tao's Lemma 3.2.

\begin{lemma}\label{epsilon-remove}
Assume that $n_k,t_k,\mu$ are as in Theorem \ref{global-restriction}, and let $R>0$ be large enough. Suppose that $\{I_1,\dots,I_M\}$ is a {\em sparse} collection 
of $R$-cubes, in the sense that their centers $x_1,\dots,x_M$ are $R^B M^B$-separated for some large enough constant $B$
(depending on $\beta$).
Then for any $f$ supported on $\bigcup_{j=1}^M I_j$, we have
$$\|\widehat{f}\|_{L^2(d\mu)} \lesssim_\epsilon R^{C\epsilon} \|f\|_{L^{p'}(\RR)}.$$
Here and below, the constant $C$ in the exponent may depend on $B$, and may change from line to line, but is independent of $R,M,\epsilon$, or $f$. 
\end{lemma}

\begin{proof} We follow the outline of Tao's argument, with modifications necessary to adapt it to our setting.
We first note the following estimate: if $f$ is supported in an $R$-cube $J$ with $R\leq N_\ell$, then
\begin{equation}\label{loc-e3}
\|\widehat{f}\|_{L^2(d\mu_\ell)} \lesssim_\epsilon R^\epsilon \|f\|_{L^{p'}(J)}.
\end{equation}
The implicit constant depends on 
$\epsilon$, but not on $f$, $R$ or $J$. This is proved as in Lemma \ref{random cantor measure} and Corollary \ref{localized-restriction}, except that we do not take the limit $\ell\to\infty$ in the proof of Lemma \ref{random cantor measure}.

Let $k\in\NN$ be such that $N_k\leq R < N_{k+1}$. We have $|E_k|=T_k N_k^{-d}$; by (\ref{t vs n}), this implies that
\begin{equation}\label{range-R}
R^{\frac{2d}{p}-d-\epsilon } \lesssim_\epsilon |E_k| \lesssim_\epsilon R^{\frac{2d}{p}-d+\epsilon}.
\end{equation}
Let $f=\sum f_i \phi_i$, where supp$\, f_i\subset I_i$ and $\phi_i=\phi_{I_i}$ for a fixed Schwartz function $\phi$ such that 
$\phi\geq 0$, $\phi\geq 1$ on $[-1,1]^d$, and supp$\,\widehat{\phi}\subset [-1,1]^d$. Note that $\widehat{\phi_i}(x)=R^d \widehat{\phi}(Rx)$, and in particular $\widehat{\phi_i}$ is supported in $[-R^{-1}, R^{-1}]^d$. Then
$$
\widehat{f} = \sum \widehat{\phi_i f_i } = \sum \widehat{\phi_i}* \widehat{ f_i}.
$$
By the support properties of $\widehat{\phi_i}$, for $x\in E$ we actually have
$$
\widehat{f} (x) = \sum \widehat{\phi_i}* (\widehat{ f_i}\, \one_{E_k}) (x),
$$
where we abuse the notation slightly and use $E_k$ to denote both the $k$-th stage set from the Cantor iteration and a $CN_k^{-1}$-neighbourhood of $E$. This is harmless since either set can be covered by a bounded number of translates of the other.

We claim that the following holds: for all $r\in [1,2]$, and for any collection of functions $F_1,\dots,F_M\in L^2(\RR^d)$, we have
\begin{equation}\label{claim-r}
\Big\| \sum_i F_i * \widehat{\phi_i} \Big\|_{L^2(\mu)}^r
\lesssim_\epsilon R^{C\epsilon}  |E_k|^{-r/2} \sum_i \|F_i \|_{L^2(\RR^d)}^r.
\end{equation}

Assuming the claim (\ref{claim-r}), we complete the proof of the lemma as follows. Let $F_i=\widehat{f_i}\,\one_{E_k}$, and
observe that 
$$
|E_k|^{-r/2} \|F_i \|_{L^2(\RR^d)}^r = \|\widehat{f_i} \|_{L^2(\mu_k)}^r.
$$
Applying (\ref{claim-r}) to $F_i$ with $r=p'$, and then using (\ref{loc-e3}), we get
\begin{align*}
\|\widehat{f}\|_{L^2(d\mu)}^{p'}
=\Big\| \sum_i F_i * \widehat{\phi_i} \Big\|_{L^2(\mu)}^{p'}
&\lesssim_\epsilon R^{C\epsilon}  |E_k|^{-p'/2} \sum_i \|F_i \|_{L^2(\RR^d)}^{p'}\\
&\lesssim_\epsilon R^{C\epsilon}  \sum_i \|\widehat{f_i} \|_{L^2(\mu_k)}^{p'}\\
&\lesssim_\epsilon R^{C\epsilon}  \sum_i \|f_i \|_{L^{p'}(\RR^d)}^{p'}\\
&\approx R^{C\epsilon}  \| f \|_{L^{p'}(\RR^d)}^{p'}
\end{align*}
as required. 

It remains to prove (\ref{claim-r}). We will do so by interpolating between $r=1$ and $r=2$. For $r=1$, it suffices to prove that
for each $i$, 
\begin{equation}\label{claim-r1}
\| F_i * \widehat{\phi_i} \|_{L^2(\mu)}
\lesssim_\epsilon R^{C\epsilon}  |E_k|^{-1/2} \|F_i \|_{L^2(\RR^d)},
\end{equation}
since this implies (\ref{claim-r}) by triangle inequality. To prove (\ref{claim-r1}), we interpolate between $L^1$ and $L^\infty$ estimates. First, we have by Fubini's theorem
\begin{align*}
\| F_i * \widehat{\phi_i} \|_{L^1(\mu)}&\leq \iint |F_i(x-y)|\, |\widehat{\phi_i}(y)| dy\, d\mu(x)
\\
&= \int |F_i(u)|\, \left( \int |\widehat{\phi_i}(v-u)| d\mu(v) \right) du
\\
&\lesssim R^d \sup_{x} \mu(x+[-R^{-1}, R^{-1}]^d)
 \|F_i\|_{L^1(\RR^d)}
\\
& \lesssim_\epsilon R^d R^{-\frac{2d}{p}+\epsilon} \|F_i\|_{L^1(\RR^d)}
\\
&\lesssim_\epsilon R^{C\epsilon}  |E_k|^{-1} \|F_i \|_{L^1(\RR^d)}
\end{align*}
where at the last step we used (\ref{range-R}). Interpolating this with the pointwise bound
$$
\sup_x |F_i * \widehat{\phi_i}(x)| \leq \|F_i\|_{L^\infty(\RR^d)} \, \|\widehat{\phi_i}\|_{L^1(\RR^d)} 
\lesssim  \|F_i\|_{L^\infty(\RR^d)}
$$
we get (\ref{claim-r1}).

To complete the argument, we need to prove (\ref{claim-r}) with $r=2$. Define the functions $g_i$ via $\widehat{g_i}=F_i$, so that 
$\|g_i\|_{L^2(\RR)}= \|F_i\|_{L^2(\RR)}$ and $F_i * \widehat{\phi_i} =\widehat{g_i\phi_i}$. We thus need to prove that 
\begin{equation}\label{claim-r2}
\Big\| \sum_i \widehat{g_i \phi_i} \Big\|_{L^2(\mu)}^2
\lesssim_\epsilon R^{C\epsilon}  |E_k|^{-1} \sum_i \|g_i \|_{L^2(\RR^d)}^2.
\end{equation}

By translational invariance and the rapid decay of $ \widehat{\phi_i}$, it suffices to prove (\ref{claim-r2}) with $ \widehat{\phi_i}$ replaced by $\one_{B_i}$. 
Let $\mathcal{R}$ be the operator $\mathcal{R}(h)= \widehat{h}|_{E}$. By Corollary \ref{localized-restriction}, $\mathcal{R}$ is a bounded operator from  
$L^{p'}(J)$ to $L^2(\mu)$ for any bounded cube $J$ (with norm depending on $J$). We have to prove that
\begin{equation*}
\Big\| \mathcal{R}(\sum_i g_i \one_{I_i} )\Big\|_{L^2(\mu)}
\lesssim_\epsilon R^{C\epsilon}  |E_k|^{-1/2} \Big( \sum_i \|g_i \|^2_{L^2(\RR^d)}\Big)^{1/2}
\end{equation*}
By the $T^*T$ argument, it suffices to prove that
\begin{equation*}
\left( \sum_i 
\left\| \one_{I_i} \mathcal{R}^* \mathcal{R} \Big(\sum_j g_j \one_{I_j} \Big)\right\|_{L^2(\RR^d)}^2
\right)^{1/2}
\lesssim_\epsilon R^{C\epsilon}  |E_k|^{-1} \Big( \sum_i \|g_i \|^2_{L^2(\RR^d)}\Big)^{1/2}
\end{equation*}
By Schur's test, this follows from
\begin{equation}\label{claim-r99}
\sup_j \sum_i 
\left\| \one_{I_i} \mathcal{R}^* \mathcal{R} \one_{I_j} h \right\|_{L^2(\RR^d)}
\lesssim_\epsilon R^{C\epsilon}  |E_k|^{-1}  \|h \|_{L^2(\RR^d)}.
\end{equation}
We claim that
\begin{equation}\label{claim-i=j}
\left\| \one_{I_i} \mathcal{R}^* \mathcal{R} \one_{I_i} h \right\|_{L^2(\RR^d)}
\lesssim_\epsilon R^{C\epsilon}  |E_k|^{-1}  \|h \|_{L^2(\RR^d)}
\end{equation}
and
\begin{equation}\label{claim-i not j}
\left\| \one_{I_i} \mathcal{R}^* \mathcal{R} \one_{I_j} h \right\|_{L^2(\RR^d)}
\lesssim  M^{-2}  \|h \|_{L^2(\RR^d)},\ \ i\neq j,
\end{equation}
with constants independent of $i,j$.
Together, these two imply (\ref{claim-r99}).

We first prove (\ref{claim-i=j}), By Lemma \ref{localized-restriction}, H\"older's inequality, and by  (\ref{range-R}), we have
\begin{align*}
\left\|  \mathcal{R} \one_{I_i} h \right\|_{L^2(d\mu)}
&\lesssim_\epsilon R^{\epsilon} \| \one_{I_i} h \|_{L^{p'}(\RR^d)}\\
&\lesssim_\epsilon R^{C\epsilon} R^{\frac{d}{2}-\frac{d}{p}} \|  h \|_{L^{2}(\RR^d)} \\
&\lesssim_\epsilon R^{C\epsilon}  |E_k|^{-1/2}  \|h \|_{L^2(\RR^d)}.
\end{align*}
This implies (\ref{claim-i=j}) by the $T^*T$ argument with a fixed $i$.

For $i\neq j$, we note that $\mathcal{R}^* \mathcal{R}h = h*\widehat{\mu}$, so that 
$\one_{I_i} \mathcal{R}^* \mathcal{R} \one_{I_j}$ is an integral operator with the kernel
$$
K_{ij}(x,y)=\one_{I_i}(x)\one_{I_j}(y) \widehat{\mu}(x-y).
$$
By (\ref{mu-decay}), $\int |K(x,y)|dy\lesssim |I_j| M^{-B\beta} R^{-B\beta}\lesssim M^{-2}$ if $B$ was chosen large enough depending on $\beta$. 
The claimed estimate (\ref{claim-i not j}) now follows from Schur's test.

\end{proof}

%
%


\section{Fourier decay}\label{Fourier decay}

To complete the proof of Theorem \ref{main-thm}, it now suffices to prove that the Cantor set in Section \ref{cantor set} can be constructed so that (\ref{baby steps}), (\ref{baby epsilons}), and (\ref{main-e2}) all hold. Since (\ref{main-e2}) implies (\ref{mu-decay}), the restriction estimate (\ref{restriction}) will follow from Proposition \ref{global-restriction}.

In all our intermediate results so far, it did not matter how the $\Lambda(p)$ alphabet sets $S_{k,a}$ were chosen,
as long as they had the prescribed cardinalities. Here, however, it is crucial to randomize the choice of $S_{k,a}$.

\begin{theorem}
Let $\{n_k\}_{k\in\NN}$ and $\{t_k\}_{k\in\NN}$ be two deterministic sets of integers such that 
(\ref{baby steps}), (\ref{baby epsilons}), (\ref{large t}) all hold, and that $\Sigma_k$ is non-empty for each $k$ (as provided by Bourgain's theorem). Let $\{\mu_k\}_{k\in\NN\cup\{0\}}$ be a sequence of random measures 
on $[0,1]^d$ such that:
\begin{itemize}
\item $\mu_0=\one_{[0,1]^d}$,
\item $\mu_1,\mu_2,\dots$ are constructed inductively via the iterative process described in Section \ref{cantor set},
\item for each $k\in\NN$, the sets $S_{k,a}$ are chosen randomly and independently from $\Sigma_k$, with probability distribution such that 
\begin{equation}\label{small expectations}
\mathbb{E}(\mu_{k}(x)|E_n)=\mu_{k-1}(x)\ \ \ \forall x\in[0,1]^d.
\end{equation}
\end{itemize}
Then the limiting Cantor measure $\mu$ almost surely obeys all conclusions of Theorem \ref{main-thm}. 
\end{theorem}

An example of a random construction of $\mu_k$ that meets the condition (\ref{small expectations}) is as follows.
Choose $n_k$ and $t_k$ as indicated in the theorem (recall that for (\ref{baby epsilons}) to hold, it suffices to assume (\ref{slow growth})). 
For each $k\in\NN$, choose a $\Lambda(p)$ set $B_k\subset [n_k]^d$ such that $|B_k|=t_k\geq c_0 n_k^{2d/p}$ and
(\ref{lambda-p}) holds with $n=n_k$, for some $c_0,C(p)$ independent of $k$. Let
$$ 
\mathcal{B}_k=\{  B_{k,v}:\   v\in [n_k]^d\},\ \  
B_{k,v}\subset[n_k]^d,\ 
B_{k,v}= v+B_k \mod (n_k\ZZ)^d  $$
Then $\mathcal{B}_k\subset \Sigma_k(n_k,t_k,c_0,2^dC(p))$, since any function supported on $B_{k,v}$ is a sum of at most $2^d$ functions supported on translates of $B_k$. 

Set $A_0=\{0\}$. Let now $k\geq 1$, and assume that $A_{k-1}$ has been constructed. For each $a\in A_{k-1}$, choose
a random $v(k,a)\in[n_k]^d$ so that $\mathbb{P}(v(k,a)=v)=n_k^{-d}$ for each $v\in[n_k]^d$ and the choices are independent for different $a\in A_{k-1}$. Let $S_{k,a}=B_{k,v(a)}$, a ``random translate" of $B_k$, and continue the construction as in Section \ref{cantor set}. Then (\ref{small expectations}) holds by translational averaging, and all other assumptions of the theorem hold with $C(p)$ replaced by $2^d C(p)$.

Instead of using random translates of a single set $B_k$ for each $k$, we could choose a set $B_{k,a}\in
\Sigma_k$ for each $a\in A_{k-1}$, then let $S_{k,a}= B(k,a)+v(k,a) \mod (n_k\ZZ)^d$, where $v(k,a)$ is a random translation vector in $[n_k]^d$ as above, chosen independently of $B_{k,a}$ and independently of the choices made for all other $a$. Bourgain's theorem \cite{bourg-lambdaP} shows that a generic subset of $[N]^d$ of size about $N^{-2d/p}$
is a $\Lambda(p)$ set, so that $\Sigma_k$ (with an appropriate choice of $c_0$ and $C(p)$) should be large for most values of $t_j$ in the indicated range, providing many sets available for the construction.
Other variants are possible.

We now turn to the proof of the theorem.

\begin{proof}
By Lemma \ref{dimension}, $E_\infty=\hbox{supp\,}\mu$ has Hausdorff dimension $\alpha$, and $\mu$ obeys (\ref{main-e1})
for all $0<\gamma<\alpha$. The Fourier decay estimate (\ref{main-e2}) is proved by a calculation almost identical to that in \cite[Section 6]{LP} for a special case in dimension 1, and in \cite[Theorem 14.1]{shmerkin-suomala2014} (see also
\cite[Theorem 4.2]{shmerkin-suomala2016}) for more general measures in higher dimensions. The proof in \cite[Theorem 14.1]{shmerkin-suomala2014} can be followed here almost word for word, except for the trivial changes in parameters to allow a variable sequence $\{n_j\}$ instead of a constant one (see e.g. \cite{chen}). It is easy to check that the proof goes through as long as
$$
\log N_{k+1}\lesssim_\epsilon  N_k^\epsilon,\ \ k\in\NN.
$$
Since $\log N_{k+1}=\log N_k +\log n_{k+1}\lesssim_\epsilon N_k^\epsilon + n_{k+1}^\epsilon$, this is a weaker condition than (\ref{baby epsilons}).

Finally, the restriction estimate (\ref{restriction}) holds by Proposition \ref{global-restriction} and by (\ref{main-e2}).

\end{proof}

\vskip.5in

\section{Acknowledgements}
This work was started while the first author was visiting the Institute for Computational and Experimental Research in Mathematics (ICERM).
The first author was supported by the NSERC Discovery Grant 22R80520. 
We would like to thank Laura Cladek, Semyon Dyatlov, Larry Guth, Mark Lewko, Pablo Shmerkin and Josh Zahl for helpful conversations.


\bigskip

\noindent{\sc Department of Mathematics, UBC, Vancouver,
B.C. V6T 1Z2, Canada}

\smallskip

\noindent{\it  ilaba@math.ubc.ca}

\medskip

\noindent{\sc Department of Mathematics, MIT, Cambridge, MA 02139, USA}

\smallskip

\noindent{\it  hongwang@mit.edu}

\end{document}